\documentclass[11pt]{article}

\usepackage{amsmath, amssymb, amsthm, calc}  

\title{On Diophantine exponents and \\ Khintchine's transference principle.
                              \thanks{ This research was  supported by
                              RFBR (grant $\textup N^{\circ}$ 09--01--00371a) and
                              the grant of the President of Russian Federation
                              $\textup N^\circ$ MK--1226.2010.1.
                             }}
\author{Oleg\,N.\,German}
\date{}

\theoremstyle{definition}
\newtheorem{definition}{Definition}

\theoremstyle{remark}

\theoremstyle{plain}
\newtheorem{theorem}{Theorem}
\newtheorem{lemma}{Lemma}

\newtheorem{proposition}{Proposition}
\newtheorem{corollary}{Corollary}

\newtheorem{classic}{Theorem}
\newtheorem{classicprime}[classic]{Theorem}

\newenvironment{classic_prime}
               {\begin{classicprime}}
               {\end{classicprime}}

\DeclareMathOperator{\vol}{vol}

\DeclareMathOperator{\spanned}{span}

\renewcommand{\vec}[1]{\mathbf{#1}}
\renewcommand{\geq}{\geqslant}
\renewcommand{\leq}{\leqslant}
\renewcommand{\phi}{\varphi}

\newcommand{\e}{\varepsilon}

\newcommand{\R}{\mathbb{R}}
\newcommand{\Z}{\mathbb{Z}}

\newcommand{\La}{\Lambda}

\newcommand{\cL}{\mathcal{L}}

\newcommand{\cS}{\mathcal{S}}
\newcommand{\cB}{\mathcal{B}}

\newcommand{\Gl}{\textup{GL}}

\newcommand{\tr}[1]{{#1}^\intercal}

\begin{document}

  \maketitle

  \begin{abstract}
     In this paper we improve estimates of Jarn\'{\i}k and Apfelbeck for uniform Diophantine exponents of transposed systems of linear forms and generalize to the case of an arbitrary system the estimates of Laurent and Bugeaud for individual exponents. The method proposed also gives a better constant in Mahler's transference theorem.
  \end{abstract}

  \section{History and main results} \label{sec:intro}

  Consider a system of linear equations
  \begin{equation} \label{eq:the_system}
    \Theta\vec x=\vec y
  \end{equation}
  with $\vec x\in\R^m$, $\vec y\in\R^n$ and
  \[ \Theta=
     \begin{pmatrix}
       \theta_{11} & \cdots & \theta_{1m} \\
       \vdots & \ddots & \vdots \\
       \theta_{n1} & \cdots & \theta_{nm}
     \end{pmatrix},\qquad
     \theta_{ij}\in\R. \]
  Let us denote by $\tr\Theta$ the transposed matrix and consider the corresponding ``transposed'' system
  \begin{equation} \label{eq:the_transposed_system}
    \tr\Theta\vec y=\vec x,
  \end{equation}
  where, as before, $\vec x\in\R^m$ and $\vec y\in\R^n$. Integer approximations to the solutions of the systems \eqref{eq:the_system} and \eqref{eq:the_transposed_system} are closely connected, which is reflected in a large variety of so called \emph{transference theorems}. Most of them deal with the corresponding asymptotics in terms of Diophantine exponents.

  \begin{definition} \label{def:beta}
    The supremum of real numbers $\gamma$, such that there are infinitely many $\vec x\in\Z^m$, $\vec y\in\Z^n$ satisfying the inequality
    \[ |\Theta\vec x-\vec y|_\infty\leq|\vec x|_\infty^{-\gamma}, \]
    where $|\cdot|_\infty$ denotes the sup-norm in the corresponding space,
    is called the \emph{individual Diophantine exponent} of $\Theta$ and is denoted by $\beta(\Theta)$.
  \end{definition}

  \begin{definition} \label{def:alpha}
    The supremum of real numbers $\gamma$, such that for each $t$ large enough there are $\vec x\in\Z^m$, $\vec y\in\Z^n$ satisfying the inequalities
    \begin{equation*} 
      0<|\vec x|_\infty\leq t,\qquad|\Theta\vec x-\vec y|_\infty\leq t^{-\gamma},
    \end{equation*}
    is called the \emph{uniform Diophantine exponent} of $\Theta$ and is denoted by $\alpha(\Theta)$.
  \end{definition}
  
  The aim of the current paper is to prove some new inequalities connecting the quantities $\alpha(\Theta)$, $\alpha(\tr\Theta)$, $\beta(\Theta)$, $\beta(\tr\Theta)$, which generalize or refine the existing ones. We do not go into too much details in the history of the question and refer the interested reader to wonderful recent surveys by Waldschmidt \cite{waldschmidt} and Moshchevitin \cite{moshchevitin}.

  \subsection{Uniform exponents}

  Our first objective is to investigate the relation between $\alpha(\Theta)$ and $\alpha(\tr\Theta)$. When $n=m=1$ the numbers $\alpha(\Theta)$ and $\alpha(\tr\Theta)$ obviously coincide (and are actually equal to $1$, see \cite{jarnik_tiflis}). In the case $n=1$, $m=2$ they also determine one another. Jarn\'{\i}k \cite{jarnik_tiflis} proved the following remarkable

  \begin{classic} \label{t:jarnik_equality}
    If $n=1$, $m=2$, then
    \begin{equation} \label{eq:jarnik_equality}
      \alpha(\Theta)^{-1}+\alpha(\tr\Theta)=1.
    \end{equation}
  \end{classic}

  Jarn\'{\i}k \cite{jarnik_tiflis} noticed that for $n=1$, $m>2$, $\alpha(\Theta)$ and $\alpha(\tr\Theta)$ are no longer related by an equality, at least he showed that in the extreme case $\alpha(\Theta)=\infty$ we can have $\alpha(\tr\Theta)$ equal to any given number in the interval $[(m-1)^{-1},1]$. However, he proved in the case $n=1$ that $\alpha(\Theta)$ and $\alpha(\tr\Theta)$ satisfy certain inequalities:

  \begin{classic} \label{t:jarnik_inequalities}
    $(i)$ We always have
    \begin{equation} \label{eq:jarnik_inequalities_a_la_khintchine}
      \frac{\alpha(\Theta)}{(m-1)\alpha(\Theta)+m}\leq\alpha(\tr\Theta)\leq\frac{\alpha(\Theta)-m+1}{m}\,.
    \end{equation}
    $(ii)$ If $\alpha(\Theta)>m(2m-3)$, then
    \[ \alpha(\tr\Theta)\geq\frac{1}{m-1}\left(1-\frac{1}{\alpha(\Theta)-2m+4}\right). \]
    $(iii)$ If $\alpha(\tr\Theta)>(m-1)/m$, then
    \[ \alpha(\Theta)\geq m-2+\frac{1}{1-\alpha(\Theta)}\,. \]
  \end{classic}

  Theorem \ref{t:jarnik_inequalities} was later generalized by Apfelbeck \cite{apfelbeck} to the case of arbitrary $n$, $m$:

  \begin{classic} \label{t:apfelbeck_inequalities}
    $(i)$ We always have
    \begin{equation} \label{eq:apfelbeck_inequalities_a_la_dyson}
      \alpha(\tr\Theta)\geq\frac{n\alpha(\Theta)+n-1}{(m-1)\alpha(\Theta)+m}\,.
    \end{equation}
    $(ii)$ If $m>1$ and $\alpha(\Theta)>(2(m+n-1)(m+n-3)+m)/n$, then
    \[ \alpha(\tr\Theta)\geq\frac 1m
       \left(n+\frac{n(n\alpha(\Theta)-m)-2n(m+n-3)}{(m-1)(n\alpha(\Theta)-m)+m-(m-2)(m+n-3)}\right). \]
  \end{classic}

  Notice that the inequalities \eqref{eq:jarnik_inequalities_a_la_khintchine} and \eqref{eq:apfelbeck_inequalities_a_la_dyson} look very much the same as the corresponding inequalities for $\beta(\Theta)$ and $\beta(\tr\Theta)$ (see Theorems \ref{t:khintchine_transference} and \ref{t:dyson_transference} below). The reason is that they are proved with the same technique, which almost neglects the ``uniform'' nature of $\alpha(\Theta)$.

  In the current paper we improve Theorems \ref{t:jarnik_inequalities} and \ref{t:apfelbeck_inequalities}. We prove

  \begin{theorem} \label{t:my_inequalities}
    For all positive integers $n$, $m$, not equal simultaneously to $1$, 
    we have
    \begin{equation} \label{eq:my_inequalities_cases}
      \alpha(\tr\Theta)\geq
      \begin{cases}
        \dfrac{n-1}{m-\alpha(\Theta)},\quad\ \ \text{ if }\ \alpha(\Theta)\leq1, \\
        \dfrac{n-\alpha(\Theta)^{-1\vphantom{\big|}}}{m-1},\quad\text{ if }\ \alpha(\Theta)\geq1.
      \end{cases}
    \end{equation}
  \end{theorem}

  Here $\alpha(\Theta)$ and $\alpha(\tr\Theta)$ are a priori allowed to attain the value $+\infty$, which gives sense to the inequalities \ref{eq:my_inequalities_cases} in case one of the denominators happens to be equal to $0$.

  Each statement concerning $\Theta$ and $\tr\Theta$ is obviously invariant under the swapping of pairs $(n,\Theta)$ and $(m,\tr\Theta)$. Therefore, if we fix $n$ and $m$, such that $n\leq m$, $m\neq1$, the two inequalities in \eqref{eq:my_inequalities_cases} split into four ones. At the same time it follows from Minkowski's convex body theorem that
  \begin{equation} \label{eq:alpha_geq_m_over_n}
    \alpha(\Theta)\geq m/n\quad\text{ and }\quad\alpha(\tr\Theta)\geq n/m,
  \end{equation}
  whence we see that one of the four inequalities mentioned above vanishes. Thus, we get the following reformulation of Theorem \ref{t:my_inequalities}:

  \begin{theorem} \label{t:my_inequalities_reformulated}
    For all positive integers $n$, $m$, $1\leq n\leq m$, $m\neq1$, we have
    \begin{align}
      \alpha(\tr\Theta) & \geq\frac{n-\alpha(\Theta)^{-1}}{m-1}\,,
      \label{eq:my_inequalities_1} \\
      \alpha(\Theta)^{-1} & \leq\frac{n-\alpha(\tr\Theta)}{m-1}\,,\qquad\,\text{ if }\ \alpha(\tr\Theta)\leq1,
      \label{eq:my_inequalities_2} \\
      \alpha(\Theta) & \geq\frac{m-\alpha(\tr\Theta)^{-1}}{n-1}\,,\quad\text{ if }\ \alpha(\tr\Theta)\geq1.
      \label{eq:my_inequalities_3}
    \end{align}
  \end{theorem}

  The case $n=1$ is worth considering separately, for in this case we have $\alpha(\tr\Theta)\leq1$ (see \cite{jarnik_tiflis}), i.e. \eqref{eq:my_inequalities_3} can hold only if $\alpha(\tr\Theta)=1$, which is already contained in \eqref{eq:my_inequalities_2}. Thus, we get

  \begin{theorem} \label{t:my_inequalities_linear_form}
    If $\,n=1$ and $m\neq1$, then
    \begin{align}
      \alpha(\tr\Theta)  \geq\frac{1-\alpha(\Theta)^{-1}}{m-1}\,,
      \label{eq:my_inequalities_linear_form_1} \\
      \alpha(\Theta)^{-1}  \leq\frac{1-\alpha(\tr\Theta)}{m-1}\,.
      \label{eq:my_inequalities_linear_form_2}
    \end{align}
  \end{theorem}

  Theorem \ref{t:my_inequalities_linear_form} with $m=2$ obviously implies Theorem \ref{t:jarnik_equality}.
  It is also not difficult, though it takes some calculation, to see that Theorem \ref{t:my_inequalities_linear_form} implies Theorem \ref{t:jarnik_inequalities} and that Theorem \ref{t:my_inequalities} implies Theorem \ref{t:apfelbeck_inequalities}.

  \subsection{Individual exponents}

  Our second result concerns the relation between $\beta(\Theta)$ and $\beta(\tr\Theta)$. In the case $n=1$ we have the classical Khintchine's transference theorem (see \cite{khintchine_palermo}):

  \begin{classic} \label{t:khintchine_transference}
    If $n=1$, then
    \begin{equation} \label{eq:khintchine_transference}
      \frac{\beta(\Theta)}{(m-1)\beta(\Theta)+m}\leq\beta(\tr\Theta)\leq\frac{\beta(\Theta)-m+1}{m}\,.
    \end{equation}
  \end{classic}

  These inequalities cannot be improved (see \cite{jarnik_1936_1} and \cite{jarnik_1936_2}) if only $\beta(\Theta)$ and $\beta(\tr\Theta)$ are considered. Stronger inequalities can be obtained if $\alpha(\Theta)$ and $\alpha(\tr\Theta)$ are also taken into account. The corresponding result for $n=1$ belongs to Laurent and Bugeaud (see \cite{laurent_2dim}, \cite{bugeaud_laurent_ndim}). They proved the following

  \begin{classic} \label{t:bugeaud_laurent}
    If $n=1$, then
    \begin{equation} \label{eq:bugeaud_laurent}
      \frac{(\alpha(\Theta)-1)\beta(\Theta)}{((m-2)\alpha(\Theta)+1)\beta(\Theta)+(m-1)\alpha(\Theta)}\leq
      \beta(\tr\Theta)\leq
      \frac{(1-\alpha(\tr\Theta))\beta(\Theta)-m+2-\alpha(\tr\Theta)}{m-1}\,.
    \end{equation}
  \end{classic}

  It is easily verified with the help of the inequalities $\alpha(\Theta)\geq m$ and $\alpha(\tr\Theta)\geq 1/m$ valid for $n=1$ that Theorem \ref{t:bugeaud_laurent} refines Theorem \ref{t:khintchine_transference}.

  Theorem \ref{t:khintchine_transference} was generalized to the case of arbitrary $n$, $m$ by Dyson \cite{dyson} (a simpler proof was later obtained by Khintchine \cite{khintchine_dobav}):

  \begin{classic} \label{t:dyson_transference}
    For all $n$, $m$, not equal simultaneously to $1$,
    \begin{equation} \label{eq:dyson_transference}
      \beta(\tr\Theta)\geq\frac{n\beta(\Theta)+n-1}{(m-1)\beta(\Theta)+m}\,.
    \end{equation}
  \end{classic}

  In the current paper we prove an analogous refinement of Theorem \ref{t:dyson_transference}:

  \begin{theorem} \label{t:loranoyadenie}
    For all positive integers $n$, $m$, not equal simultaneously to $1$, we have three inequalities
    \begin{align}
      & \beta(\tr\Theta)\geq
      \frac{n\beta(\Theta)+n-1}{(m-1)\beta(\Theta)+m}\,,
      \label{eq:loranoyadenie_1} \\
      & \beta(\tr\Theta)\geq
      \frac{(n-1)(1+\beta(\Theta))-(1-\alpha(\Theta))}{(m-1)(1+\beta(\Theta))+(1-\alpha(\Theta))}\,, \label{eq:loranoyadenie_2} \\
      & \beta(\tr\Theta)\geq
      \frac{(n-1)(1+\beta(\Theta)^{-1})-(\alpha(\Theta)^{-1}-1)}{(m-1)(1+\beta(\Theta)^{-1})+(\alpha(\Theta)^{-1}-1)}
      \,. \label{eq:loranoyadenie_3}
    \end{align}
  \end{theorem}

  It follows from the inequalities $\beta(\Theta)\geq\alpha(\Theta)\geq m/n$ that \eqref{eq:loranoyadenie_2} is stronger than \eqref{eq:loranoyadenie_3} if and only if $\alpha(\Theta)<1$. The inequality \eqref{eq:loranoyadenie_1} coincides with \eqref{eq:dyson_transference}, and it is stronger than both \eqref{eq:loranoyadenie_2} and \eqref{eq:loranoyadenie_3} if and only if
  \[ \alpha(\Theta)<\min\left(\frac{(m-1)\beta(\Theta)+m}{n+m-1}\,,\,
     \frac{(n+m-1)\beta(\Theta)}{(n-1)+n\beta(\Theta)}\right), \]
  which is never the case if $n=1$ or $m=1$, since $\alpha(\Theta)\geq m/n$. Furthermore, if $n=1$, then $\alpha(\Theta)\geq1$ and \eqref{eq:loranoyadenie_3} becomes strongest and gives the lower bound in \eqref{eq:bugeaud_laurent}. If $m=1$, then $\alpha(\Theta)\leq1$ and \eqref{eq:loranoyadenie_2} becomes strongest and we can obtain from it the upper bound in \eqref{eq:bugeaud_laurent}, if we substitute $m$ by $n$ and $\Theta$ by $\tr\Theta$. Thus, Theorem \ref{t:loranoyadenie} generalizes Theorem \ref{t:bugeaud_laurent} and refines Theorem \ref{t:dyson_transference}.

  To clear up the division into cases with respect to both $\alpha(\Theta)$ and $\beta(\Theta)$ we give the following Proposition \ref{prop:dominions}. We leave it without proof, for the only difficulty in it is calculation and does not involve any nontrivial observations, except for the inequalities $\beta(\Theta)\geq\alpha(\Theta)\geq m/n$.

  \begin{proposition} \label{prop:dominions}
  $(i)$ If $m=1$, then $1/n\leq\alpha(\Theta)\leq1$ and for all $\beta(\Theta)\geq\alpha(\Theta)$ \eqref{eq:loranoyadenie_2} is not weaker than \eqref{eq:loranoyadenie_1} and \eqref{eq:loranoyadenie_3}.

  $(ii)$ Let $m\neq1$ and $m/n\leq\alpha(\Theta)\leq1$. If
  \[ \alpha(\Theta)\leq\beta(\Theta)\leq\frac{(d-1)\alpha(\Theta)-m}{m-1}\,, \]
  then \eqref{eq:loranoyadenie_2} is not weaker than \eqref{eq:loranoyadenie_1} and \eqref{eq:loranoyadenie_3}. Otherwise, \eqref{eq:loranoyadenie_1} is strongest.

  $(iii)$ Let $m\neq1$ and $1<\alpha(\Theta)<(d-1)/n$. If
  \[ \alpha(\Theta)\leq\beta(\Theta)\leq\frac{(n-1)\alpha(\Theta)}{d-1-n\alpha(\Theta)}\,, \]
  then \eqref{eq:loranoyadenie_3} is not weaker than \eqref{eq:loranoyadenie_1} and \eqref{eq:loranoyadenie_2}. Otherwise, \eqref{eq:loranoyadenie_1} is strongest.

  $(iv)$ Let $m\neq1$ and $\alpha(\Theta)\geq(d-1)/n$. Then for all $\beta(\Theta)\geq\alpha(\Theta)$
  \eqref{eq:loranoyadenie_3} is not weaker than \eqref{eq:loranoyadenie_1} and \eqref{eq:loranoyadenie_2}.  \end{proposition}

  \subsection{Transference theorem}

  Our third result deals with a strong version of Khintchine's transference theorem. One of the strongest ones belongs to Mahler (see \cite{mahler_casopis_linear}, \cite{mahler_matsbornik_dyson}, \cite{cassels}):

  \begin{classic} \label{t:mahler_transference}
    If there are $\vec x\in\Z^m$, $\vec y\in\Z^n$, such that
    \begin{equation} \label{eq:mahler_transference_hypothesis}
      0<|\vec x|_\infty\leq X,\qquad|\Theta\vec x-\vec y|_\infty\leq U,
    \end{equation}
    then there are $\vec x\in\Z^m$, $\vec y\in\Z^n$, such that
    \begin{equation} \label{eq:mahler_transference_statement}
      0<|\vec y|_\infty\leq Y,\qquad|\tr\Theta\vec y-\vec x|_\infty\leq V,
    \end{equation}
    where
    \[ Y=(d-1)\big(X^mU^{1-m}\big)^{\frac{1}{d-1}},\quad
       V=(d-1)\big(X^{1-n}U^n\big)^{\frac{1}{d-1}},\quad\text{ and }\quad d=n+m. \]
  \end{classic}
  
  Notice that if we define numbers $\beta_1$ and $\beta_2$ by the equalities $U=X^{-\beta_1}$, $V=Y^{-\beta_2}$, then it follows from \eqref{eq:mahler_transference_hypothesis} and \eqref{eq:mahler_transference_statement} that
  \[ \beta_2=\frac{n\beta_1+(n-1)-\varkappa}{(m-1)\beta_1+m+\varkappa}\,,\qquad
     \text{ where }\varkappa=\frac{(d-1)\ln(d-1)}{\ln X}\,, \]
  which obviously implies Theorem \ref{t:dyson_transference}.

  In this paper we improve Theorem \ref{t:mahler_transference}. Namely, we substitute the factor $d-1$ by a smaller factor tending to $1$ as $d\to\infty$ (see Theorem \ref{t:my_transference} below). Of course, it does not affect the Diophantine exponents.

  \subsection{Arbitrary functions}

  Considering only exponents when investigating the asymptotic behaviour of some quantity does not allow to detect any intermediate growth. It appears, however, that the methods used in the current paper are delicate enough to work not only with the Diophantine exponents, but with arbitrary functions satisfying some natural growth conditions. In this Section we formulate the corresponding statements (Theorems \ref{t:my_inequalities_f} and \ref{t:loranoyadenie_f}) and derive from them Theorems \ref{t:my_inequalities} and \ref{t:loranoyadenie}. We also give Corollaries \ref{cor:log} and \ref{cor:exp_log} as examples of how to detect intermediate growth.


  Let $\psi:\R_+\to\R_+$ be an arbitrary function. By analogy with Definitions \ref{def:beta}, \ref{def:alpha}, we give the following

  \begin{definition} \label{def:beta_f}
    We call $\Theta$ \emph{individually $\psi$-approximable} (or, simply, \emph{$\psi$-approximable}), if there are infinitely many $\vec x\in\Z^m$, $\vec y\in\Z^n$ satisfying the inequality
    \[ |\Theta\vec x-\vec y|_\infty\leq\psi(|\vec x|_\infty). \]
  \end{definition}

  \begin{definition} \label{def:alpha_f}
    We call $\Theta$ \emph{uniformly $\psi$-approximable}, if for each $t$ large enough there are $\vec x\in\Z^m$, $\vec y\in\Z^n$ satisfying the inequalities
    \begin{equation*} 
      0<|\vec x|_\infty\leq t,\qquad|\Theta\vec x-\vec y|_\infty\leq \psi(t).
    \end{equation*}
  \end{definition}

  Clearly, $\beta(\Theta)$ (resp. $\alpha(\Theta)$) equals the supremum of real numbers $\gamma$, such that $\Theta$ is individually (resp. uniformly) $t^{-\gamma}$-approximable.

  
  The statements we are about to give all involve the concept of the inverse function. In case $\psi$ is invertible (as a map from $\R_+$ to $\R_+$) we shall denote the corresponding inverse function by $\psi^-$. Then $\psi^-(\psi(t))=\psi(\psi^-(t))=t$ for all $t>0$. 
  
  \begin{theorem} \label{t:my_inequalities_f}
    Let $\phi:\R_+\to\R_+$ be an arbitrary function, such that
    \begin{equation} \label{eq:phi_hypothesis}
      t^n\phi(t)^{m-1}\to+\infty\quad\text{ as }\quad t\to+\infty.
    \end{equation}
    Let $\psi:\R_+\to\R_+$ be an arbitrary invertible decreasing function, such that for all $t$ large enough one of the following two conditions holds:
    
    $(i)$ $t\psi(t)$ is non-increasing and satisfies the inequality
    \begin{equation} \label{eq:phi_psi_hypothesis}
      \psi(\Delta_dt^n\phi(t)^{m-1})\leq(c\Delta_dt)^{-1},
    \end{equation}
    
    $(ii)$ $t\psi(t)$ is non-decreasing and satisfies the inequality
    \begin{equation} \label{eq:phi-_psi-_hypothesis}
      \psi^-(\Delta_dt^{n-1}\phi(t)^m)\leq(c\Delta_d\phi(t))^{-1},
    \end{equation}
    where $d=n+m$, $c=\sqrt{2d(d-1)}$, and $\Delta_d$ is defined by \eqref{eq:Delta_d_definition}\footnote{Notice that due to Corollary \ref{cor:1_and_sqrt_2} from Section \ref{sec:section_dual} we have $\Delta_d$ sandwiched between $\sqrt{1/d}$ and $\sqrt{2/d}$.}.
    
    Let $\Theta$ be uniformly $\psi$-approximable. Then $\tr\Theta$ is uniformly $\phi$-approximable.
  \end{theorem}
  
  To derive Theorem \ref{t:my_inequalities} from Theorem \ref{t:my_inequalities_f} let us set $\psi(t)=t^{-\delta}$, $\phi(t)=\varkappa t^{-\gamma}$, where $\delta\leq\alpha(\Theta)$ is a positive real number, however close to $\alpha(\Theta)$ (if $\alpha(\Theta)<+\infty$, then we can just set $\delta=\alpha(\Theta)$),
  \begin{equation} \label{eq:gamma}
    \gamma=
      \begin{cases}
        \dfrac{n-1}{m-\delta},\qquad\qquad\qquad\qquad\qquad\qquad \text{ if }\ \delta<1, \\ \vphantom{\frac{|}{}}
        \text{an arbitrarily large real number},\quad\text{ if }\ \delta=1,\ m=1, \\
        \dfrac{n-\delta^{-1}}{m-1},\qquad\qquad\qquad\qquad\qquad\ \quad\text{ if }\ \delta\geq1,\ m\neq1, 
      \end{cases}
  \end{equation}
  and
  \[ \varkappa=
     \begin{cases}
       \left(c^\delta\Delta_d^{\delta-1}\right)^\frac{1}{m-\delta},\quad\ \ \,\text{ if }\ \delta<1, \\ \vphantom{\frac{\big|}{|}}
       1,\qquad\qquad\qquad\quad\ \,\text{ if }\ \delta=1,\ m=1, \\
       \left(c\Delta_d^{1-\delta}\right)^\frac{1}{(m-1)\delta},\quad\ \text{ if }\ \delta\geq1,\ m\neq1.
     \end{cases} \]
  The relation \eqref{eq:phi_hypothesis} is easily verified to be true. Furthermore, $t\psi(t)=t^{1-\delta}$ is either non-increasing, or non-decreasing, depending on whether $\delta\geq1$ or $\delta\leq1$. Besides that, we have \eqref{eq:phi_psi_hypothesis} and \eqref{eq:phi-_psi-_hypothesis} valid with equalities instead of inequalities. Hence, taking into account that $\varkappa$ does not depend on $t$, we see that $\alpha(\tr\Theta)\geq\gamma$, which implies Theorem \ref{t:my_inequalities}.

  Slightly modifying the argument given above one can see that if we know $\Theta$ to have the functional order of uniform approximation logarithm times better than $t^{-\alpha(\Theta)}$, then almost the same can be said about $\tr\Theta$. It is formalized in the following

  \begin{corollary} \label{cor:log}
    Let $\alpha(\Theta),\alpha(\tr\Theta)<+\infty$ (which excludes the case $\alpha(\Theta)=m=1$) and let $\Theta$ be uniformly $(\ln t)^{-1}t^{-\alpha(\Theta)}$-approximable. Then $\tr\Theta$ is uniformly $g(t)t^{-\gamma}$-approximable, where
    \[ g(t)= 
        \begin{cases}
          \left(\gamma c^{-\alpha(\Theta)}\Delta_d^{1-\alpha(\Theta)}\ln t\right)^{-\frac{1}{m-\alpha(\Theta)}},
          \qquad\qquad\qquad\,\text{ if }\ \alpha(\Theta)<1, \\
          \vphantom{\frac{\Big|}{}}
          (1+\e)
          \left(\alpha(\Theta)^{-1}c^{-1}\Delta_d^{\alpha(\Theta)-1}\ln t\right)^{-\frac{1}{(m-1)\alpha(\Theta)}},
          \quad\text{ if }\ \alpha(\Theta)\geq1,
        \end{cases} \]
    \begin{equation*}
      \gamma=
        \begin{cases}
          \dfrac{n-1}{m-\alpha(\Theta)},\quad\ \ \text{ if }\ \alpha(\Theta)<1, \\ 
          \vphantom{\frac{\Big|}{}}
          \dfrac{n-\alpha(\Theta)^{-1}}{m-1},\quad\text{ if }\ \alpha(\Theta)\geq1,
        \end{cases}
    \end{equation*}
    $\e>0$ is however small and the constants $c$ and $\Delta_d$ are as in Theorem \ref{t:my_inequalities_f}.
  \end{corollary}

  As was mentioned above, we always have the inequalities $\alpha(\Theta)\geq m/n$ and $\alpha(\tr\Theta)\geq n/m$. It is well known that $\alpha(\Theta)=m/n$ if and only if $\alpha(\tr\Theta)=n/m$ (it can also be seen from Theorem \ref{t:my_inequalities}). Hence, if $\alpha(\Theta)=m/n$ and the functional order of uniform approximation for $\Theta$ is $\ln t$ times better than $t^{-\alpha(\Theta)}$, then by Corollary \ref{cor:log} the functional order of uniform approximation for $\tr\Theta$ is $O(\ln^\delta t)$ times better than $t^{-\alpha(\tr\Theta)}$, where
  \[ \delta=\frac{n}{m(\max(n,m)-1)}\,. \]


%
  
  If $\alpha(\Theta)=+\infty$ (which can only happen if $m\neq1$), then by Theorem \ref{t:my_inequalities} we have $\alpha(\tr\Theta)\geq\frac{n}{m-1}$, but this does not mean that $\tr\Theta$ is uniformly $t^{-\frac{n}{m-1}}$-approximable, we can only conclude that for every $\e>0$ it is uniformly $t^{-\frac{n}{m-1}+\e}$-approximable. However, if we can estimate the functional order of approximation for $\Theta$, Theorem \ref{t:my_inequalities_f} gives more specific information. For instance, one can easily derive

  \begin{corollary} \label{cor:exp_log}
    Let $\Theta$ be uniformly $e^{-t}$-approximable. Then $\tr\Theta$ is uniformly $f(t)$-approximable, where
    \[ f(x)=\Delta_d^{-\frac{1}{m-1}}t^{-\frac{n}{m-1}}\ln(c\Delta_dt)^{\frac{1}{m-1}} \]
    and the constants $c$ and $\Delta_d$ are as in Theorem \ref{t:my_inequalities_f}.
  \end{corollary}
  
  Now let us turn to a functional analogue of Theorem \ref{t:loranoyadenie}.
  
  \begin{theorem} \label{t:loranoyadenie_f}
    Let $\phi,\psi:\R_+\to\R_+$ be arbitrary invertible decreasing functions, such that $\phi(t)\geq\psi(t)$ for all $t>0$.
    Set
    \[ \begin{array}{ll}
      f_1(t)=\left(ct^m\phi(t)\psi(t)^{1-m}\right)^\frac{1}{d-2}, & \quad
      f_{-1}(t)=\left(ct^{2-m}\phi^-(t)\psi^-(t)^{m-1}\right)^\frac{1}{d-2}, \\
      g_1(t)=\left(ct^{2-n}\phi(t)\psi(t)^{n-1}\right)^\frac{1}{d-2}, & \quad
      g_{-1}(t)=\left(ct^n\phi^-(t)\psi^-(t)^{1-n}\right)^\frac{1}{d-2}, \vphantom{^\frac{\big|}{}}
    \end{array} \]
    where $d=n+m$ and $c=\sqrt{2d(d-1)}$.

    Let $\Theta$ be $\psi$-approximable and uniformly $\phi$-approximable. Then the following two statements hold: 

    $(i)$ if $f_1$ is increasing and invertible, then $\tr\Theta$ is $(g_1\circ f_1^-)$-approximable;

    $(ii)$ if $f_{-1}$ is decreasing and invertible, then $\tr\Theta$ is $(g_{-1}\circ f_{-1}^-)$-approximable.
  \end{theorem}

  To derive Theorem \ref{t:loranoyadenie} from Theorem \ref{t:loranoyadenie_f} let us set $\psi(t)=t^{-\delta}$, $\phi(t)=t^{-\gamma}$, where $\delta$ and $\gamma$ are positive real numbers, however close to $\beta(\Theta)$ and $\alpha(\Theta)$, respectively, such that $\delta\leq\beta(\Theta)$, $\gamma\leq\alpha(\Theta)$, $\gamma\neq1$.
  
  Then for $k=\pm1$
  \begin{align*}
    f_k(t)= & \left(ct^{k(m-1)(1+\delta^k)+1-\gamma^k}\right)^\frac{1}{d-2}, \\
    g_k(t)= & \left(ct^{-k(n-1)(1+\delta^k)+1-\gamma^k}\right)^\frac{1}{d-2}.
  \end{align*}
  The only case when $f_k(t)$ is not invertible, is the case $m=1$, $\gamma=1$, which we have excluded by the choice of $\gamma$. Thus, $f_k(t)$ is invertible, and it can be easily verified that $f_1$ is increasing, $f_{-1}$ is decreasing, and
  \[ g_k(f_k^-(t))=c^{\textstyle\frac{k(n-1)(1+\delta^k)+\gamma^k}{d-2}}
                   t^{-\textstyle\frac{(n-1)(1+\delta^k)-k(1-\gamma^k)}
                                     {(m-1)(1+\delta^k)+k(1-\gamma^k)}}. \]
  Hence, taking into account that $c$ does not depend on $t$, we see that 
  \[ \beta(\tr\Theta)\geq\frac{(n-1)(1+\delta^k)-k(1-\gamma^k)}{(m-1)(1+\delta^k)+k(1-\gamma^k)}, \]
  which implies \eqref{eq:loranoyadenie_2} and \eqref{eq:loranoyadenie_3}.
  
  Notice also, that in case $\alpha(\Theta),\beta(\Theta)<+\infty$ and $\alpha(\Theta)\neq1$ we could have simply set $\delta=\beta(\Theta)$, $\gamma=\alpha(\Theta)$.
  
  As for \eqref{eq:loranoyadenie_1}, it does not need proof, for it is the very statement of Dyson's Theorem \ref{t:dyson_transference}. Thus, Theorem \ref{t:loranoyadenie} indeed follows from Theorem \ref{t:loranoyadenie_f}.

  The rest of the paper is organized as follows. Sections \ref{sec:d_dimensional}, \ref{sec:exterior}, \ref{sec:section_dual} and \ref{sec:the_lemma} are devoted to the description of the main constructions lying in the basis of the proofs, in Section \ref{sec:transference} we improve Theorem \ref{t:mahler_transference}, in Sections \ref{sec:loranoyadenie} and \ref{sec:uniform_exponents} we prove Theorems \ref{t:loranoyadenie_f} and Theorem \ref{t:my_inequalities_f}, respectively, and in Section \ref{sec:3D} we prove a refined $3$-dimensional version of Theorem \ref{t:my_inequalities_f} with better constants and compare it with an analogous theorem by Jarn\'{\i}k.

  To finish this Section we notice that Definitions \ref{def:beta}, \ref{def:alpha}, \ref{def:beta_f}, \ref{def:alpha_f}, as well as Theorem \ref{t:mahler_transference}, are given in terms of the sup-norms in $\R^n$ and $\R^m$. However, our results are in the essence valid for arbitrary norms in these spaces. The choice of particular norms affects only some of the constants and does not affect any of the exponents.

  \section{From $\R^n$ and $\R^m$ to $\R^{n+m}$} \label{sec:d_dimensional}

  Let us set $d=n+m$. Given $\vec x\in\R^m$, $\vec y\in\R^n$, we shall write $\vec z=(\vec x,\vec y)\in\R^d$ , and vice versa, given $\vec z\in\R^d$, we shall denote its first $m$ coordinates as $\vec x$ and the last $n$ ones as $\vec y$. This naturally embeds the system \eqref{eq:the_system} into $\R^d$.

  Let us denote by $\pmb\ell_1,\ldots,\pmb\ell_m,\vec e_{m+1},\ldots,\vec e_d$ the columns of the matrix
  \[ T=
     \begin{pmatrix}
       E_m & 0 \\
       -\Theta & E_n
     \end{pmatrix}, \]
  where $E_m$ and $E_n$ are the corresponding unity matrices, and by $\vec e_1,\ldots,\vec e_m,\pmb\ell_{m+1},\ldots,\pmb\ell_d$ the columns of
  \[ T'=
     \begin{pmatrix}
       E_m & \tr\Theta \\
       0 & E_n
     \end{pmatrix}. \]
  We obviously have $T\tr{(T')}=E_d$, so the bases $\pmb\ell_1,\ldots,\pmb\ell_m,\vec e_{m+1},\ldots,\vec e_d$ and $\vec e_1,\ldots,\vec e_m,\pmb\ell_{m+1},\ldots,\pmb\ell_d$ are dual. Therefore, the subspaces
  \[ \cL^m=\spanned_\R(\pmb\ell_1,\ldots,\pmb\ell_m),\qquad
     \cL^n=\spanned_\R(\pmb\ell_{m+1},\ldots,\pmb\ell_d) \]
  are orthogonal. More than that, $\cL^m=(\cL^n)^\perp$ and
  \[ \cL^m=\{ \vec z\in\R^d \,|\, \langle\pmb\ell_{m+i},\vec z\rangle=0,\ i=1,\ldots,n \},\qquad
     \cL^n=\{ \vec z\in\R^d \,|\, \langle\pmb\ell_j,\vec z\rangle=0,\ j=1,\ldots,m \}. \]
  Thus, $\cL^m$ coincides with the space of solutions of the system $\Theta\vec x=-\vec y$, and $\cL^n$ coincides with that of the system $\tr\Theta\vec y=\vec x$. It can be easily verified that changing the sign in \eqref{eq:the_system} while preserving it for the transposed system does not affect any of the definitions and theorems given in Section \ref{sec:intro}.

  For all positive $h$ and $r$ let us define parallelepipeds
  \begin{align*}
    M_{h,r}=\Big\{ \vec z\in\R^d \,\Big|\,
    |\langle\pmb\ell_{m+i},\vec z\rangle| & \leq h,\ \ i=1,\ldots,n, \\
    |\langle\vec e_j,\vec z\rangle| & \leq r,\ \ j=1,\ldots,m \Big\}
  \end{align*}
  and
  \begin{align*}
    \widehat M_{h,r}=\Big\{ \vec z\in\R^d \,\Big|\,
    |\langle\vec e_{m+i},\vec z\rangle| & \leq h,\ \ i=1,\ldots,n, \\
    |\langle\pmb\ell_j,\vec z\rangle| & \leq r,\ \ j=1,\ldots,m \Big\}.
  \end{align*}
  In these terms Definitions \ref{def:beta_f} and \ref{def:alpha_f} for $\Theta$ and $\tr\Theta$ can be reformulated as follows.

  \begin{proposition} \label{prop:definitions_reformulated}
    Given an arbitrary function $\psi:\R_+\to\R_+$, the following statements hold: \\
    $(i)$ $\Theta$ is $\psi$-approximable, if there are $t\in\R$, however large, such that
    \begin{equation} \label{eq:individual_approximability_reformulated}
      M_{\psi(t),t}\cap\Z^d\backslash\{\vec 0\}\neq\varnothing.
    \end{equation}
    $(ii)$ $\Theta$ is uniformly $\psi$-approximable, if \eqref{eq:individual_approximability_reformulated} holds for all $t\in\R$ large enough. \\
    $(iii)$ $\tr\Theta$ is $\psi$-approximable, if there are $t\in\R$, however large, such that
    \begin{equation} \label{eq:uniformal_approximability_reformulated}
      \widehat M_{t,\psi(t)}\cap\Z^d\backslash\{\vec 0\}\neq\varnothing.
    \end{equation}
    $(iv)$ $\tr\Theta$ is uniformly $\psi$-approximable, if \eqref{eq:uniformal_approximability_reformulated} holds for all $t\in\R$ large enough.
%
  \end{proposition}

  Analogically, Theorem \ref{t:mahler_transference} turns into

  \setcounter{classic}{\value{classic}-1}

  \begin{classic_prime} \label{t:mahler_transference_reformulated}
    If
    \[ M_{\scriptscriptstyle U,X}\cap\Z^d\backslash\{\vec 0\}\neq\varnothing, \]
    then
    \[ \widehat M_{\scriptscriptstyle Y,V}\cap\Z^d\backslash\{\vec 0\}\neq\varnothing, \]
    where
    \[ Y=(d-1)\big(X^mU^{1-m}\big)^{\frac{1}{d-1}},\quad
       V=(d-1)\big(X^{1-n}U^n\big)^{\frac{1}{d-1}}. \]
%
  \end{classic_prime}

  Preserving the essence, this point of view gives us an interpretation of the problem in terms of approaching to a subspace and to its orthogonal complement by integer points. Such a setting is classical and allows using many powerful techniques. One of the main tools here is the following observation (see Theorem 1 of Chapter VII, Section 3 of \cite{hodge_pedoe}):

  \begin{proposition} \label{prop:plucker_and_dual_plucker}
    Let $\cL$ be a $k$-dimensional subspace of $\R^d$ and let $\vec v_1,\ldots,\vec v_d$ be linearly independent vectors, such that $\vec v_1,\ldots,\vec v_k\in\cL$, $\vec v_{k+1},\ldots,\vec v_d\in\cL^\perp$. Let also $A\in\Gl_d(\R)$. Then $(A^\ast)^{-1}\cL^\perp=(A\cL)^\perp$ and
    \[ \frac{|\vec v_1\wedge\ldots\wedge\vec v_k|}{|\vec v_{k+1}\wedge\ldots\wedge\vec v_d|}=(\det A)^{-1}
       \frac{|A\vec v_1\wedge\ldots\wedge A\vec v_k|}{|(A^\ast)^{-1}\vec v_{k+1}\wedge\ldots
       \wedge(A^\ast)^{-1}\vec v_d|}\,. \]
  \end{proposition}

  Here $|\vec v_1\wedge\ldots\wedge\vec v_k|$ denotes the Euclidean norm of $\vec v_1\wedge\ldots\wedge\vec v_k\in\wedge^k(\R^d)$, the latter being a Euclidean $\binom dk$-dimensional space with a natural orthonormal basis $\{\vec e_{i_1}\wedge\ldots\wedge\vec e_{i_k}\}$. Due to the Cauchy-Binet formula, $|\vec v_1\wedge\ldots\wedge\vec v_k|$ is equal to the (non-oriented) $k$-dimensional volume of the parallelepiped spanned by $\vec v_1,\ldots,\vec v_k$, i.e.
  \begin{equation} \label{eq:cauchy_binet}
    |\vec v_1\wedge\ldots\wedge\vec v_k|=
    \det(\langle\vec v_i,\vec v_j\rangle)^{1/2}.
  \end{equation}

  Notice that in the basis $\vec e_1,\ldots,\vec e_d$ we have $A^\ast=\tr A$, that is in this basis $(A^\ast)^{-1}$ coincides with the cofactor matrix of $A$. Taking into account that $T'$ is exactly the cofactor matrix of $T$, we see that given two orthogonal subspaces of $\R^d$ described with the help of $\pmb\ell_1,\ldots,\pmb\ell_d$, we can apply Proposition \ref{prop:plucker_and_dual_plucker} to get orthogonal subspaces described with the help of $\vec e_1,\ldots,\vec e_d$, with certain information about volumes in these subspaces preserved.

  \section{Determinants of orthogonal integer lattices} \label{sec:exterior}

%

  For each lattice $\La\subset\R^d$ we denote by $\det\La$ its determinant, or its covolume. That is, if $\vec v_1,\ldots,\vec v_k$ is any basis of $\La$, then $\det\La$ is equal to the $k$-dimensional volume of the parallelepiped spanned by $\vec v_1,\ldots,\vec v_k$, which in its turn is equal to $|\vec v_1\wedge\ldots\wedge\vec v_k|$, due to \eqref{eq:cauchy_binet}.

  The following Proposition \ref{prop:covolumes} seems to be classical, but we decided to give it with proof since we didn't find one in the literature.

  \begin{proposition} \label{prop:covolumes}
    Let $\cL$ be a $k$-dimensional subspace of $\R^d$, such that the lattice $\La=\cL\cap\Z^d$ has rank $k$. Set $\La^\perp=\cL^\perp\cap\Z^d$. Then
    \[ \det\La=\det\La^\perp. \]
  \end{proposition}

  \begin{proof}
    Let $\vec v_1,\ldots,\vec v_k$ be a basis of $\La$ and $\vec v_{k+1},\ldots,\vec v_d$ --- that of $\La^\perp$. Let $\vec e_1,\ldots,\vec e_d$ be as in Section \ref{sec:d_dimensional}. Then
    \[ \vec v_1\wedge\ldots\wedge\vec v_k=
       \sum_{1\leq i_1<\ldots<i_k\leq d}V^{i_1\ldots i_k}\vec e_{i_1}\wedge\ldots\wedge\vec e_{i_k}, \]
    \[ \vec v_{k+1}\wedge\ldots\wedge\vec v_d=
       \sum_{1\leq i_{k+1}<\ldots<i_d\leq d}V^{i_{k+1}\ldots i_d}\vec e_{i_{k+1}}\wedge\ldots\wedge\vec e_{i_d}, \]
    where the coefficients $V^{i_1\ldots i_k}$ are coprime integers, as well as the coefficients $V^{i_{k+1}\ldots i_d}$. On the other hand,
    it follows from Theorem 1 of Chapter VII, Section 3 of \cite{hodge_pedoe} that the numbers $|V^{i_1\ldots i_k}|$ should be proportional to $|V^{i_{k+1}\ldots i_d}|$.
    This can only be if $|V^{i_1\ldots i_k}|=|V^{i_{k+1}\ldots i_d}|$ for each set of pairwise distinct indices $i_1,\ldots,i_d$. Thus, due to \eqref{eq:cauchy_binet} we have
    \[ \det\La=|\vec v_1\wedge\ldots\wedge\vec v_k|=
           |\vec v_{k+1}\wedge\ldots\wedge\vec v_d|=\det\La^\perp. \]
  \end{proof}

  \section{Section-dual set} \label{sec:section_dual}

  Let $\cS^{d-1}$ denote the Euclidean unit sphere in $\R^d$. For each measurable set $M\subset\R^d$ and each $\vec e\in \cS^{d-1}$ we denote by $\vol_{\vec e}(M)$ the $(d-1)$-dimensional volume of the intersection of $M$ and the hyperspace orthogonal to $\vec e$.

  \begin{definition}
    Let $M$ be a measurable subset of $\R^d$. We call the set
    \[ M^\wedge=\{\, \lambda\vec e\ |\ \vec e\in\cS^{d-1},\ 0\leq\lambda\leq2^{1-d}\vol_{\vec e}(M)\, \} \]
    \emph{section-dual} for $M$.
  \end{definition}

  Proposition \ref{prop:covolumes} and Minkowski's convex body theorem immediately imply

  \begin{lemma} \label{l:section_dual_at_work}
    Let $M$ be convex and $\vec 0$-symmetric. Let $M^\wedge\cap\Z^d\backslash\{\vec 0\}\neq\varnothing$. Then $M\cap\Z^d\backslash\{\vec 0\}\neq\varnothing$.
  \end{lemma}

  Let us now describe some properties of the set $M^\wedge$. Obviously, $M^\wedge$ is always symmetric with the origin as the center of symmetry.

  \begin{lemma} \label{l:section_dual_is_convex}
    If $M$ is convex, then so is $M^\wedge$.
  \end{lemma}

  \begin{proof}
    Let $\vec v_0$, $\vec v_1$ be arbitrary distinct non-zero points of $M^\wedge$. Consider an arbitrary point $\vec v_\lambda=(1-\lambda)\vec v_0+\lambda\vec v_1$, $0\leq\lambda\leq1$. The hyperspaces orthogonal to $\vec v_0$, $\vec v_\lambda$ and $\vec v_1$ intersect by a $(d-2)$-dimensional subspace. Hence, due to the convexity of $M$, we have $\vol_{\vec v_\lambda}(M)\geq(1-\lambda)\vol_{\vec v_0}(M)+\lambda\vol_{\vec v_1}(M)$, which means that $\vec v_\lambda\in M^\wedge$. Thus, $M^\wedge$ is convex.
  \end{proof}

  If $M$ is a compact convex and $\vec 0$-symmetric body (and this is the case we shall be interested in, more than that, in all our applications throughout the paper $M$ will be a parallelepiped), then the set $M^\wedge$ resembles a lot the set $[M]^{(d-1)}$ --- the $(d-1)$-th compound of $M$ defined by Mahler (see \cite{mahler_compound_I} or \cite{gruber_lekkerkerker}). $[M]^{(d-1)}$ is a subset of $\wedge^{d-1}(\R^d)$, and the latter is isomorphic to $\R^d$, so we can consider $[M]^{(d-1)}$ lying in the same space as $M^\wedge$. Then it can be easily verified that
  \[ M^\wedge\varsubsetneq [M]^{(d-1)}\varsubsetneq (d-1)M^\wedge. \]
  Indeed, if $S$ is the intersection of $M$ with a $(d-1)$-dimensional subspace of $\R^d$ and if $\vec v_1,\ldots,\vec v_{d-1}$ are chosen in $S$ so that $\det(\langle\vec v_i,\vec v_j\rangle)$ is maximal, then $S$ is a proper subset of the parallelepiped
  \[ \Big\{ \sum_{i=1}^{d-1}\lambda_i\vec v_i\, \Big| -1\leq\lambda_i\leq1 \Big\}. \]

  Furthermore, $M^\wedge$ behaves just like $[M]^{(d-1)}$ when $M$ is subjected to a linear transformation:

  \begin{lemma} \label{l:compound_matrix}
    Let $A$ be a non-degenerate $d\times d$ real matrix. Then $(AM)^\wedge=A'(M^\wedge)$, where $A'$ denotes the cofactor matrix of $A$.
  \end{lemma}

  \begin{proof}
%
    It suffices to apply Proposition \ref{prop:plucker_and_dual_plucker} and notice that if $S$ is a subset of a $(d-1)$-dimensional subspace spanned by $\vec v_1,\ldots,\vec v_{d-1}$, then the quotient of the $(d-1)$-dimensional volume of $S$ and $|\vec v_1\wedge\ldots\wedge\vec v_{d-1}|$ is a linear invariant.
  \end{proof}

  Denote by $\cB_\infty^d$ the unit ball in the sup-norm in $\R^d$, i.e. the cube
  \[ \Big\{ \vec x=(x_1,\ldots,x_d)\in\R^d \ \Big|\ |x_i|\leq1,\ i=1,\ldots,d\, \Big\} \]
  and set
  \begin{equation} \label{eq:Delta_d_definition}
    \Delta_d=\frac{1}{2^{d-1}\sqrt d}\vol_{d-1}\Big\{ \vec x\in\cB_\infty^d \,\Big|\, \sum_{i=1}^dx_i=0 \Big\},
  \end{equation}
  where $\vol_{d-1}(\cdot)$ denotes the $(d-1)$-dimensional Lebesgue measure.

  \begin{lemma} \label{l:cube_wedge}
    $(\cB_\infty^d)^\wedge$ contains (the convex hull of) the points with only two non-zero coordinates, which are equal to $\pm1$, and the points $(\pm\Delta_d,\ldots,\pm\Delta_d)$.
  \end{lemma}

  \begin{proof}
    The points $(\pm\Delta_d,\ldots,\pm\Delta_d)$ are obviously in $(\cB_\infty^d)^\wedge$. The volume of the section of $\cB_\infty^d$ orthogonal to the point $(1,1,0,\ldots,0)$ is equal to $2^{d-1}\sqrt2$, hence this point is also in $(\cB_\infty^d)^\wedge$. The rest is obvious.
  \end{proof}

  \begin{corollary} \label{cor:cube_wedge_cube}
    $(\cB_\infty^d)^\wedge$ contains the cube $\Delta_d\cB_\infty^d$.
  \end{corollary}

  \begin{corollary} \label{cor:cube_wedge_octahedron}
    $(\cB_\infty^d)^\wedge$ contains the set defined by the inequalities
    \begin{equation} \label{eq:cube_wedge_octahedron}
      \sum_{i=1}^d|x_i|\leq2,\qquad|x_j|\leq1,\qquad j=1,\ldots,d.
    \end{equation}
  \end{corollary}

  Let us now say a couple of words concerning the asymptotic behaviour of $\Delta_d$. We shall use it to improve the transference theorem (see Theorem \ref{t:my_transference}). Vaaler's and Ball's theorems (see \cite{vaaler}, \cite{ball}) imply the following

  \begin{proposition} \label{prop:1_and_sqrt_2}
    The volume of each $(d-1)$-dimensional central section of $\cB_\infty^d$ is bounded between $2^{d-1}$ and $2^{d-1}\sqrt2$.
  \end{proposition}

  \begin{corollary} \label{cor:1_and_sqrt_2}
    We have $\sqrt{d/2}\leq\Delta_d^{-1}\leq\sqrt d$.
  \end{corollary}

  \section{Transference theorem} \label{sec:transference}

  As it was mentioned in Section \ref{sec:intro}, the factor $d-1$ in Theorem \ref{t:mahler_transference} can be substituted by a smaller factor tending to $1$ as $d\to\infty$. This new factor is
  \[ \Delta_d^{-\frac{1}{d-1}}. \]
  It follows from Corollary \ref{cor:1_and_sqrt_2} that it is less than $d-1$ for $d\geq3$ and that
  \[ \Delta_d^{-\frac{1}{d-1}}\to1\quad\text{ as }\quad d\to\infty. \]
  Thus, taking into account Theorem \ref{t:mahler_transference_reformulated}, we see that the following Theorem \ref{t:my_transference} improves Theorem \ref{t:mahler_transference}.
  From now on we use the notations of Section \ref{sec:d_dimensional}.

  \begin{theorem} \label{t:my_transference}
    If
    \[ M_{\scriptscriptstyle U,X}\cap\Z^d\backslash\{\vec 0\}\neq\varnothing, \]
    then
    \[ \widehat M_{\scriptscriptstyle Y,V}\cap\Z^d\backslash\{\vec 0\}\neq\varnothing, \]
    where
    \begin{equation} \label{eq:transference_parameters}
       Y=\Delta_d^{-\frac{1}{d-1}}\big(X^mU^{1-m}\big)^{\frac{1}{d-1}},\quad
       V=\Delta_d^{-\frac{1}{d-1}}\big(X^{1-n}U^n\big)^{\frac{1}{d-1}}.
    \end{equation}
  \end{theorem}

  \begin{proof}
    Set
    \[ A=T\cdot
         \begin{pmatrix}
           \dfrac{E_m}{V} & 0 \\
           0 & \dfrac{E_n}{Y}
         \end{pmatrix}=
         \begin{pmatrix}
           \dfrac{E_m}{V} & 0 \\
           \dfrac{-\Theta}{V} & \dfrac{E_n}{Y}
         \end{pmatrix}. \]
    For the cofactor matrix $A'$ we obviously have
    \begin{equation*} \label{eq:cofactor_of_A}
      A'=T'\cdot
         \begin{pmatrix}
           \dfrac{E_m}{Y^nV^{m-1}} & 0 \\
           0   & \dfrac{E_n}{Y^{n-1}V^m}
         \end{pmatrix}=
         \begin{pmatrix}
           \dfrac{E_m}{Y^nV^{m-1}} & \dfrac{\tr\Theta}{Y^{n-1}V^m} \\
           0   & \dfrac{E_n}{Y^{n-1}V^m}
         \end{pmatrix}.
    \end{equation*}
    Then $\widehat M_{\scriptscriptstyle Y,V}=(A^\ast)^{-1}\cB_\infty^d$ and, in view of \eqref{eq:transference_parameters},
    \begin{align*}
      ((A')^\ast)^{-1}B_\infty^d=M_{\scriptscriptstyle Y^{n-1}V^m,Y^nV^{m-1}}=\Delta_d^{-1}M_{\scriptscriptstyle U,X}.
    \end{align*}
    Applying Corollary \ref{cor:cube_wedge_cube} and Lemma \ref{l:compound_matrix} we see that
    \[ M_{\scriptscriptstyle U,X}\subset(\widehat M_{\scriptscriptstyle Y,V})^\wedge. \]
    It remains to make use of Lemma \ref{l:section_dual_at_work}.
  \end{proof}

  It should be mentioned that Mahler derived Theorem \ref{t:mahler_transference} from a somewhat stronger result. Using his bilinear form method he actually proved that in all the inequalities \eqref{eq:mahler_transference_statement} but one the factor $d-1$ can be omitted. Then in our terms \eqref{eq:mahler_transference_statement} becomes
  \begin{equation} \label{eq:only_one_d-1}
  \begin{split}
    |\langle\vec e_{m+i},\vec z\rangle|\leq &\
    \lambda_{m+i}\big(X^mU^{1-m}\big)^{\frac{1}{d-1}},\quad i=1,\ldots,n, \\
    |\langle\pmb\ell_j,\vec z\rangle|\leq &\
    \lambda_j\big(X^{1-n}U^n\big)^{\frac{1}{d-1}},\qquad\,\ j=1,\ldots,m,
  \end{split}
  \end{equation}
  with only one of the $\lambda_k$ equal to $d-1$ and the rest of them equal to $1$.

  Such a statement does not immediately follow from Theorem \ref{t:my_transference}, but it can be easily obtained by a slight modification of its proof. Indeed, let us choose any of the $\lambda_k$ to be equal to $d-1$, denote by $M$ the parallelepiped defined by \eqref{eq:only_one_d-1}, and consider the linear operator $C$, such that $M=C\cB_\infty^d$. Then, with $C'$ denoting the cofactor matrix of $C$, we have
  \begin{align*}
    C'\cB_\infty^d=\Big\{ \vec z\in\R^d \,\Big|\,
    & |\langle\vec e_j,\vec z\rangle|\leq\frac{\mu}{\lambda_j}X,\ \ j=1,\ldots,m, \\
    & |\langle\pmb\ell_{m+i},\vec z\rangle|\leq\frac{\mu}{\lambda_{m+i}}U,\ \ i=1,\ldots,n \Big\},
  \end{align*}
  where $\mu$ is the product of all the $\lambda_k$. One of the $\mu/\lambda_k$ is equal to $1$ and all the others are equal to $d-1$. Hence due to Corollary \ref{cor:cube_wedge_octahedron} and Lemma \ref{l:compound_matrix} we have $M_{\scriptscriptstyle U,X}\subset M^\wedge$, since the facets of the polyhedron defined by \eqref{eq:cube_wedge_octahedron} parallel to coordinate hyperplanes are generalized octahedra with the radii of the inscribed spheres equal to $(d-1)^{-1/2}$. Once again, it remains to apply Lemma \ref{l:section_dual_at_work}.

  \section{The main lemma} \label{sec:the_lemma}

  In this section we prove Lemma \ref{l:M_h_r_section}, which describes the main step in all the proofs to be given in the subsequent Sections. Notice that Lemma \ref{l:M_h_r_section} is in some sense a two-dimensional analogue of Lemma \ref{l:section_dual_at_work}.

  As in Section \ref{sec:d_dimensional}, we shall denote the first $m$ coordinates of a point $\vec z\in\R^d$ as $\vec x$, and the last $n$ ones as $\vec y$. We shall also use the notations $M_{h,r}$ and $\widehat M_{h,r}$ introduced in Section \ref{sec:d_dimensional}.

  \begin{lemma} \label{l:cube_section}
    Let $\vec z_1=(\vec x_1,\vec y_1)\in\R^d$, $\vec z_2=(\vec x_2,\vec y_2)\in\R^d$. Then
    \begin{equation} \label{eq:z_wedge_z_leq_max}
      |\vec z_1\wedge\vec z_2|\leq
      \sqrt{2d(d-1)}
      \max\Big(|\vec x_1|_\infty|\vec x_2|_\infty,\ |\vec y_1|_\infty|\vec y_2|_\infty,\
      \max\big(|\vec x_1|_\infty,|\vec x_2|_\infty\big)\max\big(|\vec y_1|_\infty,|\vec y_2|_\infty\big)\Big).
    \end{equation}
  \end{lemma}

  \begin{proof}
    We have
    \[ \vec z_1\wedge\vec z_2=\sum_{1\leq i_1<i_2\leq d}V^{i_1i_2}\vec e_{i_1}\wedge\vec e_{i_2}. \]
    Let $|V^{j_1j_2}|$ be the maximal number among all the $|V^{i_1i_2}|$, $1\leq i_1<i_2\leq d$. Then
    \begin{equation} \label{eq:z_wedge_z_leq_V}
      |\vec z_1\wedge\vec z_2|\leq\sqrt{\frac{d(d-1)}{2}}|V^{j_1j_2}|.
    \end{equation}
    The value $|V^{j_1j_2}|$ is equal to the volume of the projection of the parallelogram spanned by $\vec z_1$, $\vec z_2$ onto the subspace $\spanned_\R(\vec e_{j_1}\vec e_{j_2})$. Therefore,
    \begin{equation} \label{eq:V_leq_cases}
      |V^{j_1j_2}|\leq
      \begin{cases}
        2|\vec x_1|_\infty|\vec x_2|_\infty,\qquad\qquad\qquad\qquad\qquad\qquad\quad\text{ if }j_1<j_2\leq m, \\
        2|\vec y_1|_\infty|\vec y_2|_\infty,\qquad\qquad\qquad\qquad\qquad\qquad\quad\text{ if }j_2>j_1>m, \\
        2\max\big(|\vec x_1|_\infty,|\vec x_2|_\infty\big)\max\big(|\vec y_1|_\infty,|\vec y_2|_\infty\big),\quad
        \text{ if }j_1\leq m<j_2.
      \end{cases}
    \end{equation}
    These are the only three possible cases, since $j_1<j_2$. Combining \eqref{eq:z_wedge_z_leq_V} and \eqref{eq:V_leq_cases}, we get \eqref{eq:z_wedge_z_leq_max}.
  \end{proof}

  \begin{lemma} \label{l:M_h_r_section}
    Let $h$, $r$, $h_1$, $r_1$, $h_2$, $r_2$ be positive real numbers and let $\vec v_1$, $\vec v_2$ be non-collinear points of $\Z^d$.
    Suppose that
    \begin{equation} \label{eq:v1_v2}
      \vec v_1\in M_{h_1,r_1}\,,\qquad\vec v_2\in M_{h_2,r_2}
    \end{equation}
    and
    \begin{equation} \label{eq:bound_for_the_products}
      \max\Big(r^2r_1r_2,\ h^2h_1h_2,\ hr\max(r_1,r_2)\max(h_1,h_2)\Big)\leq
      \frac{h^nr^m}{\sqrt{2d(d-1)}}\,.
    \end{equation}
    Then
    \[ \widehat M_{h,r}\cap\Z^d\backslash\{\vec 0\}\neq\varnothing. \]
  \end{lemma}

  \begin{proof}
    Let us consider the matrix
    \[ A=T\cdot
         \begin{pmatrix}
           \dfrac{E_m}{r} & 0 \\
           0 & \dfrac{E_n}{h}
         \end{pmatrix}=
         \begin{pmatrix}
           \dfrac{E_m}{r} & 0 \\
           \dfrac{-\Theta}{r} & \dfrac{E_n}{h}
         \end{pmatrix} \]
    and its inverse conjugate
    \begin{equation*}
      (A^\ast)^{-1}=(\tr A)^{-1}=(\det A)^{-1}\cdot T'\cdot
         \begin{pmatrix}
           \dfrac{E_m}{h^nr^{m-1}} & 0 \\
           0   & \dfrac{E_n}{h^{n-1}r^m}
         \end{pmatrix}=
         \begin{pmatrix}
           rE_m & h\tr\Theta \\
           0   & hE_n
         \end{pmatrix}.
    \end{equation*}
    Then $\widehat M_{h,r}=(A^\ast)^{-1}\cB_\infty^d$. For each of the points
    \[ \vec z_k=(\vec x_k,\vec y_k)=A^{-1}\vec v_k,\qquad k=1,2, \]
    we have
    \[ |\vec x_k|_\infty=r\max_{1\leq j\leq m}|\langle\vec e_j,\vec v_k\rangle|,\qquad
       |\vec y_k|_\infty=h\max_{1\leq i\leq n}|\langle\pmb\ell_{m+i},\vec v_k\rangle|. \]
    Hence in view of \eqref{eq:v1_v2} and \eqref{eq:bound_for_the_products} we get
    \begin{equation} \label{eq:max_bounded_by_det}
    \begin{aligned}
      \max\Big(|\vec x_1|_\infty|\vec x_2|_\infty,\ |\vec y_1|_\infty|\vec y_2|_\infty,\
      \max\big(|\vec x_1|_\infty,|\vec x_2|_\infty\big)\max\big(|\vec y_1|_\infty,|\vec y_2|_\infty\big)\Big) &
      \leq \\ \leq
      \max\Big(r^2r_1r_2,\ h^2h_1h_2,\ hr\max(r_1,r_2)\max(h_1,h_2)\Big) & \leq
      \frac{(\det A)^{-1}}{\sqrt{2d(d-1)}}\,.
    \end{aligned}
    \end{equation}
    Set
    \[ \cL=\spanned_\R(\vec v_1,\vec v_2)\quad\text{ and }\quad\La=\cL\cap\Z^d. \]
    Clearly, $\spanned_\Z(\vec v_1,\vec v_2)$ is a sublattice of $\La$ and its determinant is a multiple of $\det\La$. Applying Proposition \ref{prop:plucker_and_dual_plucker}, Lemma \ref{l:cube_section} and inequality \eqref{eq:max_bounded_by_det} we see that
    \begin{equation} \label{eq:max_bounded_by_det_implies}
      \frac{|\vec v_1\wedge\vec v_2|}{2^{2-d}\vol_{d-2}(\cL^\perp\cap\widehat M_{h,r})}=
      \frac{(\det A)\cdot|\vec z_1\wedge\vec z_2|}{2^{2-d}\vol_{d-2}((A^{-1}\cL)^\perp\cap\cB_\infty^d)}\leq
      \frac{1}{2^{2-d}\vol_{d-2}((A^{-1}\cL)^\perp\cap\cB_\infty^d)}\leq1.
    \end{equation}
    The latter inequality is due to Vaaler's theorem (see \cite{vaaler}), which says that the volume of any $(d-2)$-dimensional section of $\cB_\infty^d$ is bounded from below by $2^{d-2}$. Thus,
    \[ \vol_{d-2}(\cL^\perp\cap\widehat M_{h,r})\geq2^{d-2}\det\La, \]
    which, together with Proposition \ref{prop:covolumes} and Minkowski's convex body theorem, implies that $\cL^\perp\cap\widehat M_{h,r}$ contains a nonzero integer point.
  \end{proof}

  In order to make application of Lemma \ref{l:M_h_r_section} more convenient it is useful to mention the following observation.

  \begin{lemma} \label{l:sufficiency_for_M_h_r_section}
    If
    \[ \max(r_1,r_2)\max(h_1,h_2)\leq\frac{h^{n-1}r^{m-1}}{\sqrt{2d(d-1)}} \]
    and any of the equalities
    \[ r_1r_2=\frac{h^nr^{m-2}}{\sqrt{2d(d-1)}}\,,\qquad h_1h_2=\frac{h^{n-2}r^m}{\sqrt{2d(d-1)}} \]
    holds, then we have \eqref{eq:bound_for_the_products}.
  \end{lemma}

  \begin{proof}
    Everything follows from the inequality $r_1r_2h_1h_2\leq\big(\max(r_1,r_2)\max(h_1,h_2)\big)^2$.
  \end{proof}

  \section{Proof of Theorem \ref{t:loranoyadenie_f}} \label{sec:loranoyadenie}


  \begin{lemma} \label{l:the_alpha_beta_semicore_1}
    Let $t$, $\Phi$, $\Psi$ be arbitrary positive real numbers, $\Phi\geq\Psi$.
    Let $\vec v_1$, $\vec v_2$ be non-collinear integer points, such that
    \begin{equation} \label{eq:v1_v2_1}
      \vec v_1\in M_{\Phi,t}\quad\text{ and }\quad\vec v_2\in M_{\Psi,t}\,.
    \end{equation}
    Then
    \[ \widehat M_{h,r}\cap\Z^d\backslash\{\vec 0\}\neq\varnothing, \]
    where
    \[ h=\left(ct^m\Phi\Psi^{1-m}\right)^\frac{1}{d-2},\qquad
       r=\left(ct^{2-n}\Phi\Psi^{n-1}\right)^\frac{1}{d-2},\qquad
       c=\sqrt{2d(d-1)}. \]
  \end{lemma}

  \begin{proof}
    We have
    \begin{align*}
      h^{n-2}r^m & =c\Phi\Psi, \\
      h^{n-1}r^{m-1} & =ct\Phi.
    \end{align*}
    It remains to make use of the inequality $\Phi\geq\Psi$, 
    Lemma \ref{l:sufficiency_for_M_h_r_section} and Lemma \ref{l:M_h_r_section}.
  \end{proof}

  Let us derive from Lemma \ref{l:the_alpha_beta_semicore_1} the statement of Theorem \ref{t:loranoyadenie_f} with $k=1$. It follows from Proposition \ref{prop:definitions_reformulated} that for all $t$ large enough we can choose non-collinear points $\vec v_1,\vec v_2\in\Z^d$ satisfying \eqref{eq:v1_v2_1} with $\Psi=\psi(t)$ and $\Phi=\phi(t)$. Then, in case $f_1$ is increasing and invertible, we have $r=g_1(f_1^-(h))$ and $h\to+\infty$ as $t\to+\infty$. This gives the desired statement.

  To prove the statement of Theorem \ref{t:loranoyadenie_f} with $k=-1$ we need a reversed analogue of Lemma \ref{l:the_alpha_beta_semicore_1}:

  \begin{lemma} \label{l:the_alpha_beta_semicore_2}
    Let $t$, $\Phi$, $\Psi$ be arbitrary positive real numbers, $\Phi\geq\Psi$.
    Let $\vec v_1$, $\vec v_2$ be non-collinear integer points, such that
    \begin{equation} \label{eq:v1_v2_2}
      \vec v_1\in M_{t,\Phi}\quad\text{ and }\quad\vec v_2\in M_{t,\Psi}\,.
    \end{equation}
    Then
    \[ \widehat M_{h,r}\cap\Z^d\backslash\{\vec 0\}\neq\varnothing, \]
    where
    \[ h=\left(ct^{2-m}\Phi\Psi^{m-1}\right)^\frac{1}{d-2},\qquad
       r=\left(ct^n\Phi\Psi^{1-n}\right)^\frac{1}{d-2},\qquad
       c=\sqrt{2d(d-1)}. \]
  \end{lemma}

  \begin{proof}
    We have
    \begin{align*}
      h^nr^{m-2} & =c\Phi\Psi, \\
      h^{n-1}r^{m-1} & =ct\Phi.
    \end{align*}
    Once again, it remains to make use of the inequality $\Phi\geq\Psi$, 
    Lemma \ref{l:sufficiency_for_M_h_r_section} and Lemma \ref{l:M_h_r_section}.
  \end{proof}

  It follows from Proposition \ref{prop:definitions_reformulated} that for all $t$ small enough we can choose non-collinear points $\vec v_1,\vec v_2\in\Z^d$ satisfying \eqref{eq:v1_v2_2} with $\Psi=\psi^-(t)$ and $\Phi=\phi^-(t)$. Then, in case $f_{-1}$ is decreasing and invertible, we have $r=g_{-1}(f_{-1}^-(h))$ and $h\to+\infty$ as $t\to0$. This gives the statement of Theorem \ref{t:loranoyadenie_f} with $k=-1$.

  \section{Proof of Theorem \ref{t:my_inequalities_f}} \label{sec:uniform_exponents}

%

  \begin{lemma} \label{l:the_alphas_core}
    Let $\phi$, $\psi$ be as in Theorem \ref{t:my_inequalities_f} and let $h$ be an arbitrary positive real number. Set \begin{equation} \label{eq:r_is_phi_of_h}
      r=\phi(h)
    \end{equation}
    and
    \begin{equation} \label{eq:h_ast_r_ast_definition}
      h^\ast=\Delta_d\,r^mh^{n-1},\qquad r^\ast=\Delta_d\,r^{m-1}h^n,
    \end{equation}
    where $\Delta_d$ is defined by \eqref{eq:Delta_d_definition}.

    Suppose that in the interval 
    \[ \mathcal I=[r^\ast,\max(r^\ast,\psi^-(h^\ast))] \] 
    one of the conditions $(i)$, $(ii)$ of Theorem \ref{t:my_inequalities_f} holds, and 
    \begin{equation} \label{eq:M_lambda_contains_points}
      M_{\psi(t),t}\cap\Z^d\backslash\{\vec 0\}\neq\varnothing,\quad\text{ for every $t\in\mathcal I$.}
    \end{equation}
    Then
    \begin{equation} \label{eq:M_h_r_contains_points}
      \widehat M_{h,r}\cap\Z^d\backslash\{\vec 0\}\neq\varnothing.
    \end{equation}
  \end{lemma}

  \begin{proof}
    If $M_{h^\ast,r^\ast}$ contains nonzero integer points, then \eqref{eq:M_h_r_contains_points} follows from Theorem \ref{t:my_transference}. Thus, we may assume that
    \begin{equation} \label{eq:M_h_ast_r_ast_is_empty}
      M_{h^\ast,r^\ast}\cap\Z^d\backslash\{\vec 0\}=\varnothing.
    \end{equation}
    Particularly, in this case we have $r^\ast<\psi^-(h^\ast)$.
    
    Consider the minimal $\mu>1$, such that $M_{\mu h^\ast,\mu r^\ast}$ contains a nonzero integer point $\vec v$. If this point satisfies the inequality
    \[ \max_{1\leq j\leq m}|\langle\vec e_j,\vec v\rangle|\leq
       \frac hr
       \max_{1\leq i\leq n}|\langle\pmb\ell_{m+i},\vec v\rangle|, \]
    then we denote it as $\vec v_1$ and consider the minimal $\mu'\geq\mu$, such that for every $\e>0$ small enough $M_{(\mu-\e)h^\ast,\mu'r^\ast}$ contains a nonzero integer point $\vec v_2$. Otherwise, we denote it as $\vec v_2$ and, analogically, consider the minimal $\mu'\geq\mu$, such that for every $\e>0$ small enough $M_{\mu'h^\ast,(\mu-\e)r^\ast}$ contains a nonzero integer point $\vec v_1$. Set
    \begin{equation*} 
      \lambda_1=\max_{1\leq i\leq n}|\langle\pmb\ell_{m+i},\vec v_1\rangle|\quad\text{ and }\quad
      \lambda_2=\max_{1\leq j\leq m}|\langle\vec e_j,\vec v_2\rangle|.
    \end{equation*}
    Then
    \begin{equation*} 
      \max_{1\leq j\leq m}|\langle\vec e_j,\vec v_1\rangle|\leq\frac hr\lambda_1<\lambda_2\quad\text{ and }\quad
      \max_{1\leq i\leq n}|\langle\pmb\ell_{m+i},\vec v_2\rangle|\leq\frac rh\lambda_2<\lambda_1,
    \end{equation*}
    that is
    \begin{equation} \label{eq:v1_v2_bounded_lambdas}
      \vec v_1\in M_{\lambda_1,(h/r)\lambda_1}\subset M_{\lambda_1,\lambda_2}\quad\text{ and }\quad
      \vec v_2\in M_{(r/h)\lambda_2,\lambda_2}\subset M_{\lambda_1,\lambda_2}\,.
    \end{equation}
    In view of \eqref{eq:M_lambda_contains_points} and \eqref{eq:M_h_ast_r_ast_is_empty} we also have
    \begin{equation*} 
      h^\ast<\lambda_1\leq \psi(\lambda_2)<\psi(r^\ast)\quad\text{ and }\quad
      r^\ast<\lambda_2\leq \psi^-(\lambda_1)<\psi^-(h^\ast).
    \end{equation*}
    These inequalities imply that, if the condition $(i)$ of Theorem \ref{t:my_inequalities_f} holds, then by \eqref{eq:r_is_phi_of_h}, \eqref{eq:h_ast_r_ast_definition} and \eqref{eq:phi_psi_hypothesis} we have
    \begin{equation*}
      \lambda_1\lambda_2\leq\lambda_2\psi(\lambda_2)\leq r^\ast\psi(r^\ast)\leq\frac{r^{m-1}h^{n-1}}{\sqrt{2d(d-1)}}\,,
    \end{equation*}
    and if $(ii)$ holds, then by \eqref{eq:r_is_phi_of_h}, \eqref{eq:h_ast_r_ast_definition} and \eqref{eq:phi-_psi-_hypothesis} we have
    \begin{equation*}
      \lambda_1\lambda_2\leq\lambda_1\psi^-(\lambda_1)=\psi^-(\lambda_1)\psi(\psi^-(\lambda_1))\leq h^\ast\psi^-(h^\ast)\leq\frac{r^{m-1}h^{n-1}}{\sqrt{2d(d-1)}}\,.
    \end{equation*}
    Thus, in each case we have
    \begin{equation*} 
      \lambda_1\lambda_2\leq\frac{r^{m-1}h^{n-1}}{\sqrt{2d(d-1)}}\,.
    \end{equation*}
    This, together with \eqref{eq:v1_v2_bounded_lambdas}, implies \eqref{eq:bound_for_the_products}. It remains to apply Lemma \ref{l:M_h_r_section}.
  \end{proof}

  Theorem \ref{t:my_inequalities_f} is now easily derived. It follows from \eqref{eq:phi_hypothesis}, \eqref{eq:r_is_phi_of_h}, \eqref{eq:h_ast_r_ast_definition} that $r^\ast\to\infty$ as $h\to\infty$. So if $h$ is large enough, then either $(i)$, or $(ii)$ of Theorem \ref{t:my_inequalities_f} holds. Besides that, if $h$ is large enough, then by Proposition \ref{prop:definitions_reformulated} we  have  \eqref{eq:M_lambda_contains_points}. Applying Lemma \ref{l:the_alphas_core}, we get Theorem \ref{t:my_inequalities_f}.

  \section{Special case $\bf n+m=3$} \label{sec:3D}

  Jarn\'{\i}k \cite{jarnik_tiflis} derived his Theorem \ref{t:jarnik_equality} from a stronger statement for functions. In our terms it can be reformulated as follows.

  \begin{classic} \label{t:jarnik_f}
    Let $n=1$, $m=2$.
    Let $\psi:\R_+\to\R_+$ be an arbitrary invertible decreasing function, such that $f(t)=t\psi(t)$ is invertible. Let $\e$, $\delta$ be arbitrary positive real numbers. Then the following two statements hold:

    $(i)$ if $\psi(t)\leq\e t^{-2}$ for all $t$ large enough, $f(t)$ is decreasing and $\Theta$ is uniformly $\psi$-approximable, then $\tr\Theta$ is uniformly $\phi$-approximable, where
    \begin{equation} \label{eq:jarnik_3D_12}
      \phi(t)=\frac{12(1+\e+\delta)}{t}\psi^-\left(\frac{1}{t}\right);
    \end{equation}

    $(ii)$ if $\psi(t)\leq\e t^{-1/2}$ for all $t$ large enough, $f(t)$ is increasing and $\tr\Theta$ is uniformly $\psi$-approximable, then $\Theta$ is uniformly $\phi$-approximable, where
    \begin{equation} \label{eq:jarnik_3D_4}
      \phi(t)=\frac{4(1+\e+\delta)}{f^-(t/2)}\,.\ \ \
    \end{equation}
  \end{classic}

  It appears that our method gives better constants in \eqref{eq:jarnik_3D_12} and \eqref{eq:jarnik_3D_4}. In order to give in this case the best result we can, let us improve Lemma \ref{l:M_h_r_section} for $d=3$. Namely, let us replace the constant in \eqref{eq:bound_for_the_products}, which is equal in this case to $2\sqrt3$, by $2$.

  \begin{lemma} \label{l:M_h_r_section_3D}
    Let $h$, $r$, $h_1$, $r_1$, $h_2$, $r_2$ be positive real numbers and let $\vec v_1$, $\vec v_2$ be non-collinear points of $\Z^3$.
    Suppose that
    \begin{equation} \label{eq:v1_v2_3D}
      \vec v_1\in M_{h_1,r_1}\,,\qquad\vec v_2\in M_{h_2,r_2}
    \end{equation}
    and
    \begin{equation} \label{eq:bound_for_the_products_3D}
      \max\Big(r^2r_1r_2,\ h^2h_1h_2,\ hr\max(r_1,r_2)\max(h_1,h_2)\Big)\leq
      \frac{1}{2}h^nr^m\,.
    \end{equation}
    Then
    \[ \widehat M_{h,r}\cap\Z^d\backslash\{\vec 0\}\neq\varnothing. \]
  \end{lemma}

  \begin{proof}
    Since the statement of the Lemma is symmetric with respect to $n$, $m$, we may assume that $n=1$ and $m=2$. Let $A$, $\vec z_1$, $\vec z_2$, $\cL$ be as in the proof of Lemma \ref{l:M_h_r_section}. Using elementary geometric considerations it is not difficult to prove that
    \begin{equation} \label{eq:z_wedge_z_leq_max_3D}
      \frac{|\vec z_1\wedge\vec z_2|}{2^{-1}\vol_1((A^{-1}\cL)^\perp\cap\cB_\infty^d)}\leq2
      \max\Big(|\vec x_1|_\infty|\vec x_2|_\infty,\ 
      \max\big(|\vec x_1|_\infty,|\vec x_2|_\infty\big)\max\big(|\vec y_1|_\infty,|\vec y_2|_\infty\big)\Big).
    \end{equation}
    Repeating the argument proving Lemma \ref{l:M_h_r_section} with application of \eqref{eq:z_wedge_z_leq_max_3D} instead of Lemma \ref{l:cube_section}, we get \eqref{eq:max_bounded_by_det_implies}, and thus, the desired statement.
  \end{proof}
  
  We notice that an analogous improvement can be made for arbitrary $n$, $m$, if one of them is equal to $1$.
  
  Now that we have Lemma \ref{l:M_h_r_section_3D}, we can replace the constant $c$ in Theorem \ref{t:my_inequalities_f} by $2$, in case $d=3$. Thus improved, Theorem \ref{t:my_inequalities_f} implies the following

  \begin{theorem} \label{t:3D}
    Let $n=1$, $m=2$.
    Let $\psi:\R_+\to\R_+$ be an arbitrary invertible decreasing function, such that $f(t)=t\psi(t)$ is invertible. Then the following two statements hold:

    $(i)$ if $f(t)$ is decreasing and $\Theta$ is uniformly $\psi$-approximable, then $\tr\Theta$ is uniformly $\phi$-approximable, where
    \begin{equation*} 
      \phi(t)=\frac{3}{4t}\psi^-\left(\frac{2}{3t}\right);
    \end{equation*}

    $(ii)$ if $f(t)$ is increasing and $\tr\Theta$ is uniformly $\psi$-approximable, then $\Theta$ is uniformly $\phi$-approximable, where
    \begin{equation*} 
      \phi(t)=\frac{2}{3f^-(t/2)}\,.\ \ \
    \end{equation*}
  \end{theorem}

   As we claimed, Theorem \ref{t:3D} is stronger than Theorem \ref{t:jarnik_f}. Indeed, statement $(ii)$ is obviously stronger, and as for statement $(i)$, it suffices to notice that if $f(t)=t\psi(t)$ is decreasing, then $t\psi^-(t)=f(\psi^-(t))$ is increasing, since $\psi^-(t)$ is decreasing, as well as $\psi(t)$.


\vskip 10mm

\noindent
Oleg N. {\sc German} \\
Moscow Lomonosov State University \\
Vorobiovy Gory, GSP--1 \\
119991 Moscow, RUSSIA \\
\emph{E-mail}: {\fontfamily{cmtt}\selectfont german@mech.math.msu.su, german.oleg@gmail.com}

\end{document}